\documentclass{amsart}
\usepackage{latexsym}
\usepackage[nohug]{diagrams}
\usepackage{amsfonts}
\usepackage{amsmath}
\usepackage{amssymb}
\usepackage{amsthm}
\newtheorem{theorem}{Theorem}[section]

\newtheorem{proposition}[theorem]{Proposition}

\theoremstyle{definition}                               

\newtheorem*{example*}{Example}

\theoremstyle{remark} 

\newtheorem*{remark*}{Remark}
\numberwithin{equation}{section}

\DeclareMathOperator{\Pic}{\sf Pic}

\DeclareMathOperator{\End}{\sf End}
\DeclareMathOperator{\Aut}{\sf Aut}

\DeclareMathOperator{\Div}{Div}
\DeclareMathOperator{\res}{res}
\DeclareMathOperator{\Hom}{Hom}
\DeclareMathOperator{\dlog}{dlog}

\def\O{\mathcal{O}}

\def\K{\mathbb{K}}

\def\Z{\mathbb{Z}}

\def\L{\mathcal{L}}
\def\D{\mathcal{D}}
\def\c{\mathbb{C}}
\def\d{\partial}
\def\S{\Sigma}
\def\I{{\sf{Id}}}
\begin{document}
\title{Picard groups of differential operators on curves}
%
%
%
%
%
     \author{George Wilson}
     \address{Mathematical Institute, 24--29 St Giles, Oxford OX1 3LB, UK}
      \email{wilsong@maths.ox.ac.uk}
\date{October 16, 2010}
\begin{abstract}
These notes are a supplement to the first part 
of \cite{CH2}, concerning the Picard group of $\, \D(X) \,$, where 
$\, X \,$ is an affine curve. The main new fact is that the exact 
sequence of \cite{CH2} describing $\, \Pic \D \,$ is split.
\end{abstract}
\maketitle

\section{Introduction}

Let $\, X \,$ be a smooth irreducible complex affine curve.   
To make the statements that follow as simple as possible, 
we shall assume that $\, X \,$ has no nontrivial 
automorphisms, and also that the ring $\, \O(X) \,$ has no nonconstant 
units (that is the ``case of general position'': we shall remove these 
assumptions in section~\ref{gencase} below).  
Let $\, \D \equiv \D(X) \,$ be the ring of differential 
operators on $\, \O(X) \,$, and let $\, \Pic \D \,$ be the group 
of (isomorphism classes of) autoequivalences of the category of 
left $\, \D$-modules.  In \cite{CH2}, Cannings and Holland proved 
the following.
\begin{theorem}
\label{CHa}
There is an exact sequence of groups
\begin{equation}
\label{CHseq}
0 \,\to\, \Omega^1(X) \,\to\, \Pic \D \,\to\, \Pic X \,\to\, 0 \ .
\end{equation}
\end{theorem}
Here (as usual) $\, \Omega^1(X) \,$ is the additive group of regular 
differentials on $\, X \,$, and $\, \Pic X \,$ is the group of (algebraic) 
line bundles.  

Theorem~\ref{CHa} almost determines $\, \Pic \D \,$; however, there remains the 
question of the group extension in \eqref{CHseq}. 
\begin{proposition}
\label{split}
The sequence \eqref{CHseq} is split.  
\end{proposition}

This sequence takes on a more familiar appearance if 
we rephrase some of the material of \cite{CH2} in terms of line bundles with 
connection.  We denote by 
$\, \Pic^{\flat} X \,$ the group of isomorphism classes of line bundles 
with (flat algebraic) connection over $\, X \,$, so that we have 
an exact sequence of abelian groups 
\begin{equation}
\label{conseq}
0 \, \to\, \Omega^1(X) \,\to\, \Pic^{\flat} X \, \xrightarrow{p} \,  
\Pic X \,\to\, 0\ ;
\end{equation}
here the map $\, p \,$ assigns to a line bundle with connection 
the underlying line bundle, and $\, \Omega^1(X) \,$ is considered as 
the space of connections on the trivial bundle.  Now, if $\, \L \,$ 
is a line bundle on $\, X \,$, with 
space of global sections $\, \Gamma_{\L} \,$, a 
connection on  $\, \L \,$ may be viewed as a structure of left  
$\, \D$-module on $\, \Gamma_{\L} \,$, extending the given structure 
of $\, \O(X)$-module.  The functor $\, \Gamma_{\L} \otimes_{\O(X)} - \,\,$ 
then defines an element of $\, \Pic \D \,$, so we have 
a natural homomorphism
\begin{equation}
\label{chi}
\chi \,:\, \Pic^{\flat} X \,\to\, \Pic \D \ .
\end{equation}
It is easy to see that this map 
$\, \chi \,$ is injective: Cannings and Holland proved (in effect) 
the following.
\begin{theorem}
\label{CHb}
The map \eqref{chi} is an isomorphism.
\end{theorem}

Using this isomorphism, the exact sequence \eqref{CHseq} becomes 
the sequence \eqref{conseq}, which is quite amenable to study. The proof 
of Theorem~\ref{CHb} consists in combining the main results of \cite{CH1} 
and of \cite{ML}: we give the outline in section~\ref{out} below. 

Most of the discussion above still holds, {\it mutatis mutandis}, if 
$\, X \,$ is a complete (that is, projective) curve: in that 
case we have to replace the ring $\, \D(X) \,$ by the {\it sheaf} 
$\, \D_X \,$ of differential operators on $\, X \,$. I do not know 
whether Theorem~\ref{CHb} still holds in the 
complete case; however, we certainly still have the exact 
sequence \eqref{conseq}, provided we replace $\, \Pic X \,$ 
by the group $\, \Pic^0 X \,$ of line bundles 
{\it of degree zero} (since only these admit a connection). Our 
proof that the sequence is split still holds in the complete 
case.  That contrasts with the known fact (see \cite{M}) that  
$\, \Pic^{\flat} X \,$ is the {\it universal} 
extension of the Jacobian $\, \Pic^0 X \,$ by a vector group, thus 
in some sense as far as possible from being split.  The explanation 
is that this last statement considers \eqref{conseq} as an extension of 
complex algebraic groups, and the distinguished 
splitting described below is a splitting only of {\it real} Lie groups.
In the affine case, there is no natural algebraic structure on 
the groups in \eqref{conseq}: indeed, $\, \Pic X \,$ is typically a quotient 
of a torus by a countably infinite subgroup, so the only possibility 
seems to be to regard it just as a huge abstract group. 
However, if $\, X \,$ is 
obtained from a complete curve $\, \S \,$ by removing just one point, 
then $\, \Pic X \,$ is canonically identified with  $\, \Pic^0 \S \,$, 
so we are in an awkward intermediate situation.

The paper is organized as follows.  In section~\ref{dd} I review the main 
technical device used in the proof of Proposition~\ref{split}, namely the 
description of $\, \Pic^{\flat} X \,$ in terms of differentials of 
the third kind on $\, X \,$ (see \cite{M}). I give a self-contained 
account of this which is less sophisticated than the one in \cite{M}; 
it uses arguments that will be familiar to readers of \cite{CH2}.  
Section~\ref{spl} gives two constructions of splittings of the sequence 
\eqref{conseq}.  The first is purely algebraic, but involves 
an arbitrary choice of basis for an infinite-dimensional vector space. 
The second construction gives the distinguished splitting mentioned above;
however, it involves analytic considerations. From an algebraic point of 
view there seems to be no natural splitting of the sequence \eqref{CHseq}, 
which is no doubt why none was found in \cite{CH2}. In section~\ref{complete} 
I explain very briefly the claim above that for a complete curve our 
distinguished splitting is a splitting of real Lie groups; and 
section~\ref{gencase} gives the small changes needed to treat 
the case of a general affine curve (possibly with automorphisms and units). 
Finally, in section~\ref{out} we sketch the proof of the 
basic Theorem~\ref{CHb}.

\section{Differentials and divisors}
\label{dd}
Let $\, \Div X \,$ be the group of divisors on  a curve $\, X \,$, 
and let $\, \K \,$ be the field of rational functions on  $\, X \,$.
Then we have the homomorphism $\, \K^{\times} \to  \Div X \,$ 
assigning to a rational function its divisor of zeros and poles: its 
image is the subgroup $\, P \,$ of {\it principal divisors} and 
(as is very well-known) the quotient $\, \Div X/P \,$ is 
canonically identified with $\, \Pic X \,$. 

Slightly less well-known is the fact that $\, \Pic^{\flat} X \,$ 
has a similar description.  Let $\, \Omega_3(X) \,$ be the (additive) 
group of {\it differentials of the third kind} on $\, X \,$ (that is, 
rational differentials with only simple poles), and let 
$\, \Omega_3^{\Z}(X) \,$ be the subgroup of differentials with integer 
residues\footnote{Some authors include this in the definition of 
``third kind''; others call any rational differential ``of the third kind''.} 
at each pole.  There is an obvious map 
\begin{equation}
\label{reseq}
\res \,:\, \Omega_3^{\Z}(X) \,\to \,  \Div X
\end{equation}
which assigns to a differential $\, \omega \,$ the divisor 
$\, \sum_{x \in X} \,(\res_x \omega)\,x \,$: its kernel is $\, \Omega^1(X) \,$. 
We have also the map $\, \dlog : \K^{\times} \to \Omega_3^{\Z}(X) \,$: 
we shall identify its image with $\, P \,$, so that the map \eqref{reseq} 
restricts to the identity on $\, P \,$.
\begin{proposition}
\label{diff}
The quotient $\, \Omega_3^{\Z}(X)/P  \,$ can be canonically 
identified with $\, \Pic^{\flat} X \,$.
\end{proposition}

The identification is such that the diagram 
\begin{equation}
\label{diag}
\begin{diagram}[small]
0 &\rTo &\Omega^1(X) &\rTo &\Omega_3^{\Z}(X) &\rTo^{\res} &\Div X &\rTo &0\\
  &      &\dTo^{\I}   &       &\dTo          &      &\dTo     &      &\\
0 & \rTo &\Omega^1(X) &\rTo &\Pic^{\flat} X   &\rTo^p &\Pic X  &\rTo &0
\end{diagram}
\end{equation}
commutes.  Thus we can use the top sequence in this diagram to study 
the bottom one. 

The rest of this section is devoted to explaining 
Proposition~\ref{diff}: we shall concentrate on the case where $\, X \,$ 
is affine. Let us first review the notion of a 
(necessarily flat) connection on a line bundle  $\, \L \,$: roughly 
speaking it is a way of making $\, \D(X) \,$ act on (sections of) 
$\, \L \,$.  To be precise, let $\, \D_{\L} \,$ be the ring of 
differential operators on  
$\, \L \,$: it contains $\, \O(X) \,$ as the subalgebra 
of operators of degree $\, 0 \,$, and we may define 
a  connection on  $\, \L \,$ 
to be an isomorphism $\, \varphi : \D(X) \to \D_{\L} \,$ 
such that the restriction of $\, \varphi \,$ to $\, \O(X) \,$ is the 
identity.  Let us look at the special case where $\, \L \,$ is 
the trivial bundle $\, X \times \c \,$, so that 
$\, \D_{\L} \equiv \D(X) \,$.  Then a connection is just an automorphism 
of $\, \D(X) \,$ which fixes $\, \O(X) \,$.  Since $\, \D(X) \,$ is 
generated by $\, \O(X) \,$ and its derivations, such 
an automorphism  $\, \varphi \,$ is determined by its action on 
derivations $\, \d \,$; this action necessarily takes the form
\begin{equation}
\label{act}
\varphi(\d) \,=\, \d \,+\, \langle{\omega, \, \d \rangle}
\end{equation}
where $\, \omega \in \Omega^1(X) \,$ and $\, \langle \,,\,\rangle \,$ 
is the natural pairing between $\, 1$-forms and vector fields. 
In this way, connections on the trivial bundle are in 1-1 correspondence 
with regular 1-forms.

For a general line bundle $\, \L \,$, a connection is usually described 
in terms of locally defined 
1-forms as above, using local trivializations of $\, \L \,$;
however, in our algebraic situation, we can use the fact that $\, \L \,$ 
always has  a {\it rational} trivialization to describe a 
connection by a single rational differential, much as above.  
More precisely, let us fix a divisor $\, D = \sum n_x \,x \,$ in the class of 
$\, \L \,$. Corresponding to $\, D \,$ we have the fractional 
ideal of $\, \O(X) \,$
\begin{equation}
\label{ideal}
I_D \,:=\, \{ f \in \K \,:\, \nu_x(f) \geq -n_x \ \forall x \in X \}
\end{equation}
(as usual $\, \nu_x(f) \,$ is the order to which $\, f \,$ vanishes 
(or minus the order of pole) at $\, x \,$).  Then $\, I_D \,$ is 
isomorphic to the $\, \O(X)$-module $\, \Gamma_{\L} \,$ of sections 
of $\, \L \,$; indeed, choosing a divisor in the class of 
$\, \L \,$ is equivalent to choosing a (fractional) ideal 
$\, I \subset \K \,$ isomorphic to 
$\, \Gamma_{\L} \,$.  We may now identify $\, \D_{\L} \,$ 
with the algebra
$$
\D(I_D) \,:=\, \{ \theta \in \D(\K) \,:\, \theta . I_D \subseteq I_D \}\ .
$$
So a connection on $\, \L \,$ is an isomorphism 
$\, \varphi : \D(X) \to \D(I_D) \,$.  This extends uniquely to an 
automorphism $\, \varphi \,$ of $\, \D(\K) \,$ (restricting to the 
identity on $\, \K \,$), which must have the form \eqref{act}, with 
$\, \omega \in \Omega_{rat}(X) \,$ now a {\it rational} differential.  
If we change the choice of ideal $\, I \,$ by a factor 
$\, f \in \K^{\times} \,$, the corresponding differential changes by 
the gauge transformation $\, \omega \mapsto \omega + \dlog f \,$.  
Thus so far we have seen that $\, \Pic^{\flat} X \,$ embeds into the 
space $\, \Omega_{rat}(X)/{\dlog \K^{\times}} \,$.  To complete the 
explanation of Proposition~\ref{diff}, we have only to see what 
is the image of this embedding; that is, 
which rational differentials give rise to automorphisms 
of $\, \K \,$ that map $\, \D(X) \,$ onto $\, \D(I_D) \,$.
\begin{proposition}
\label{iii}
Let $\, \omega \,$ be a rational differential, $\, \varphi \,$ the 
corresponding automorphism of $\, \D(\K) \,$.  Then $\, \varphi \,$ 
maps $\, \D(X) \,$ onto $\, \D(I_D) \,$ if and only if \,{\rm (i)} 
$\, \omega \in \Omega^{\Z}_3(X) \,$ and \,{\rm (ii)} 
$\, \res \,\omega = D \,$\,.
\end{proposition}
\begin{proof}
Note first that if $\, \omega_1 \,$ and $\, \omega_2 \,$ are two 
differentials such that the corresponding automorphisms $\, \varphi_1 \,$ 
and $\, \varphi_2 \,$ both map $\, \D(X) \,$ onto $\, \D(I_D) \,$, 
then $\, \varphi_2^{-1} \varphi_1 \,$ restricts to  an automorphism of 
$\, \D(X) \,$: it follows that 
$\, \omega_1 - \omega_2 \,$ is a regular differential on $\, X \,$. 
Thus it is enough to find just one automorphism $\, \varphi \,$ which 
maps $\, \D(X) \,$ onto $\, \D(I_D) \,$, and such that the 
corresponding differential $\, \omega \,$ has the principal parts 
specified by the properties (i) and (ii) 
in Proposition~\ref{iii}.  For this, let $\, I_D^* \,$ be the fractional 
ideal inverse to $\, I_D \,$: it corresponds to the divisor $\, -D \,$.  
Since $\, I_D I_D^* \,=\, \O(X) \,$, we can choose $\, \alpha_i \in I_D \,$ 
and $\, \beta_i \in I_D^* \,$ such that 
$\, \sum  \alpha_i \beta_i = 1 \,$.  We claim that the differential 
$\, \omega \,:=\, \sum \alpha_i d \beta_i $ has the required properties.  
Indeed, the automorphism $\, \varphi \,$ corresponding to $\, \omega \,$ acts 
on derivations $\, \d \,$ of $\, \K \,$ by
\begin{align*}
\varphi(\d) &\,=\, \d \,+\, \langle \,
\sum \alpha_i d \beta_i , \, \d \,\rangle\\
            &\,=\, \d \,+\, \sum \alpha_i\, 
\langle \,d \beta_i , \, \d \,\rangle\\
            &\,=\, \d \,+\, \sum \alpha_i \d(\beta_i)\\
            &\,=\, \sum \alpha_i \d \beta_i\ ,
\end{align*}
where the last step used that $\, \sum  \alpha_i \beta_i = 1 \,$.  If now  
$\, \d \,$ is a derivation of $\, \O(X) \,$, then 
\begin{align*}
(\alpha_i \d \beta_i).I_D &\,\subseteq\, (\alpha_i \d) . \O(X)\\
                          &\,\subseteq\, \alpha_i \,\O(X)\\
                          &\,\subseteq\, I_D \ ;
\end{align*}
that is, $\, \varphi(\d) \in \D(I_D) \,$, whence $\, \varphi \,$ maps 
$\, \D(X) \,$ to $\, \D(I_D) \,$. Similarly, $\, \varphi^{-1} \,$ maps 
$\, \D(I_D) \,$ to $\, \D(X) \,$, so $\, \varphi \,$ does indeed give 
an isomorphism between these two rings. It remains to check that $\, \omega \,$ 
has the properties (i) and (ii).  Fix a point $\, x \,$ in the support 
of $\, D \,$, and let $\, z \,$ be a local parameter near $\, x \,$.  
Then near $\, x \,$ the $\, \alpha_i \,$ and $\, \beta_i \,$ have the 
form
$$
\alpha_i = a_i z^{-n_x} + \ldots\,, \quad \beta_i = b_i z^{n_x} + \ldots\,,
$$
where $\, a_i \,$ and $\, b_i \,$ are constants and 
the $\,\ldots \,$ denote higher order terms.  Thus 
$\, d \beta_i \,=\, n_x b_i z^{n_x - 1}dz + \ldots \,$\,, so 
$$
\omega \,:=\, \sum \alpha_i d \beta_i 
\,=\, (\sum a_i b_i) n_x z^{-1}dz \,+\, \ldots \ .
$$
But since $\, \sum  \alpha_i \beta_i = 1 \,$, we have 
$\, \sum  a_i b_i = 1 \,$; it follows that  
$\, \omega \,$ has a simple pole at $\, x \,$ with 
residue $\, n_x \,$.  That finishes the proof of Proposition~\ref{iii}.
\end{proof}
\section{Splitting}
\label{spl}
Let us return to the diagram \eqref{diag}.  The bottom map $\, p \,$, 
which we want to show is split, is obtained from the top map $\, \res \,$ 
by dividing out by the (common) subgroup $\, P \,$ of principal 
divisors.  Thus it is enough to construct a splitting 
$$ 
s : \Div X \to \Omega_3^{\Z}(X)
$$ 
of the top map which extends the identity map on 
$\, P \,$, for this will then descend to the 
quotient to give a splitting of the bottom map.

That is almost trivial, but not quite, since $\, P \,$ is not a direct factor 
in $\, \Div X \,$ (the quotient $\, \Pic X \,$ has elements of finite 
order, while $\, \Div X \,$ is a free abelian group).  However, 
we can consider the map of larger groups
$$
\res \,:\, \Omega_3(X) \,\to\, (\Div X) \otimes_{\Z} \c \ .
$$
These are now vector spaces, so we can certainly choose a $\c$-linear 
splitting $\, s \,$ with $\, s(p \otimes 1) = p \,$ for $\, p \in P \,$.   
The very fact that $\, s \,$ {\it is} a splitting implies that it maps 
$\, \Div X\otimes 1  \,$ into $\, \Omega_3^{\Z}(X) \,$, so we are finished.

It is more satisfactory to describe a ``natural'' splitting 
of our sequence.  Let us consider first the case of a complete curve 
$\, X \,$: in that case we interpret $\, \Div X \,$ to be group of 
divisors {\it of degree zero} on $\, X \,$.  To define $\, s \,$, for 
each such divisor $\, D \,$ we have to choose a differential with 
principal parts as prescribed by (i) and (ii) in Proposition~\ref{iii}.  
There are several ways to 
normalize a differential with prescribed principal parts: the one we need 
is to make all its periods {\it pure imaginary}.  It is clear 
that the resulting map $\, s \,$ is additive; further, if 
$\, D \,$ is the divisor of a rational function $\, f \,$, then we 
have $\, s(D) = \dlog(f) \,$; that is, $\, s \,$ extends the 
identity map on the group $\, P \,$ of principal divisors, so again 
we are finished.

If $\, X \,$ is obtained from a complete curve $\, \S \,$ by removing 
a single point, the situation is equally good.  Indeed, if 
$\, D \in \Div X \,$, then $\, D \,$ has a unique extension to a 
divisor $\, \overline{D} \,$ of degree zero on 
$\, \S  \,$; if we take the above 
normalized differential $\, s(\overline{D}) \,$ and restrict it to 
$\, X \,$, we again get a distinguished splitting of the sequence 
\eqref{conseq}. 

To extend this construction to  the general case, 
when $\, X \,$ is obtained by removing several points from $\, \S \,$, 
we need to choose some way of extending  divisors on  
$\, X \,$ to divisors of degree zero on $\, \S \,$.  For example, we could 
single out one of the points ``at infinity'' and let the others have 
multiplicity $0$ in the extension; or, more democratically, we could 
let the extended divisor have the same multiplicity at each of the points at 
infinity.  I leave the choice to the reader. 
\section{The complete case}
\label{complete}

Let us return to the case where $\, X \,$ is complete, and explain 
the claim that in that case we have a splitting of 
(finite dimensional) real Lie groups.  We take an analytic point of view. 
Recall that (as for any complex manifold) there 
is a canonical identification 
$\, \Pic^{\flat} X \simeq H^1(X, \,\c^{\times}) \,$. 
From this point of view \eqref{conseq} comes from the 
cohomology sequence of the exact sequence of analytic\footnote{These  
sheaves could be interpreted algebraically, but then the 
map $ \dlog $ would not be surjective.} sheaves
$$
0 \,\to\, \c^{\times} \,\to\, \O^{\times} \, 
\xrightarrow{\dlog} \,\Omega^1 \,\to\, 0 \ .
$$
The map 
$\, \pi : \Omega_3^{\Z} \,\to\, \Pic^{\flat} X \,$ 
also has a very simple description from this point of view:  
if we identify
$$
\Pic^{\flat} X \simeq H^1(X, \,\c^{\times}) \,\simeq\, 
\Hom(H_1(X,\,\Z), \,\c^{\times}) \ ,
$$
then $\, \pi \,$ sends a differential $\, \omega \,$ to the 
homomorphism
$$
\pi(\omega) \,:\, [c] \,\mapsto\, \exp \int_c \omega \ ,
$$
where $\, [c] \,$ is the homology class of a 1-cycle $\, c \,$.  
Clearly, $\, \omega \,$ is normalized (to have imaginary periods) 
if and only if $\, \pi(\omega) \,$ takes values in the unit circle 
$\, S^1 \subset \c^{\times} \,$.  Now, as a real Lie group 
we have the polar decomposition 
$\,\c^{\times} \simeq \mathbb{R} \times S^1 \,$, in which a pair 
$\, (\lambda, \,e^{i \theta}) \in \mathbb{R} \times S^1 \,$ 
corresponds to the complex number $\, e^{\lambda + i \theta} \,$.  
Our decomposition of $\, \Pic^{\flat} X \,$ is just the product 
of $\, 2g \,$ copies of this one.
\section{The general affine case}
\label{gencase}

We have assumed so far that $\, \O(X) \,$ has no nontrivial units, 
and that $\, X \,$ has no nontrivial automorphisms; however, neither of 
these assumptions is essential.  In the general case, let $\, U \,$ 
be the group of units in $\, \O(X) \,$. If $\, u \in U \,$, then 
$\, \dlog u \in \Omega^1(X) \,$; set
$$
\overline{\Omega} \,:=\, \Omega^1(X)/{\dlog U}\ .
$$
Recall that $\, \Pic^{\flat} X \,$ is the group of 
{\it isomorphism classes} (that is, gauge equivalence classes) of 
lines bundles with connection.  Each $\, u \in U \,$ gives an 
isomorphism of connections on the trivial bundle, changing 
the corresponding differential $\, \omega \,$ by the gauge 
transformation $\, \omega \mapsto \omega + \dlog u \,$; thus 
the space of isomorphism classes of connections on the trivial bundle 
is $\, \overline{\Omega} \,$, so in the exact sequence \eqref{conseq} 
we have to replace $\, \Omega^1(X) \, $ by $\, \overline{\Omega} \,$. 
Similarly, in the diagram \eqref{diag}, we have to change  
$\, \Omega^1(X) \, $ to $\, \overline{\Omega} \,$, and also 
$\, \Omega_3^{\Z}(X) \,$ has to be replaced by the quotient 
$\, \Omega_3^{\Z}(X)/{\dlog U} \,$.   The main part of our discussion 
is then unaffected by the presence of $\, U \,$.

The automorphism group $\, \Aut X \,$ equally easy to deal with, but more 
interesting, since it gives us extra elements of $\, \Pic \D \,$. This 
group acts 
compatibly on all the groups in the split exact sequence \eqref{conseq}, 
so we get a split exact sequence 
\begin{equation}
\label{last}
0 \,\to\, \overline{\Omega} \,\to\, \Pic^{\flat} X \rtimes \Aut X \,\to\, 
\Pic X \rtimes \Aut X \,\to\, 0 \ .
\end{equation}
Further, there is a natural inclusion
\begin{equation}
\label{hook}
\chi \,:\,  \Pic^{\flat} X \rtimes \Aut X \,\to \, \Pic \D 
\end{equation}
generalizing \eqref{chi}: Cannings and Holland show\footnote{At this point 
we have to exclude the case where $\, X \,$ is isomorphic to the 
affine line.}
that it is an isomorphism.
Inserting this isomorphism into \eqref{last}, we get the exact sequence  
of \cite{CH2}, Theorem~1.15.
\begin{remark*}
It follows from Theorem~\ref{chi} that if the group $\, \Aut X \,$  is 
trivial, then $\, \Pic \D \,$ is abelian.  On the other hand, a 
nontrivial automorphism of $\, X \,$ cannot act trivially on 
$\, \overline{\Omega} \,$ (because it does not act trivially on the 
subgroup $\, d\,\O(X) \,$). So from the exact sequence \eqref{last} 
we get the following curious fact: $\, \Pic \D \,$ is abelian if and only 
if $\, \Aut X \,$  is trivial.
\end{remark*}

\section{Outline of proof of Theorem~\ref{CHb}}
\label{out}

We first translate Theorem~\ref{CHb} into the language of bimodules 
used in \cite{CH2}.  
Recall that any autoequivalence $\, T \,$ of the  category of left 
$\,\D$-modules is given by tensoring with the invertible 
$\,\D$-bimodule $\, T(\D) \,$ 
(that is the case for any algebra; see, for example \cite{B}, p.~60 et seq.). 
Given a line bundle $\, \L \,$ 
with connection, choose an ideal $\, I \subseteq \O \,$ isomorphic 
to $\, \Gamma_{\L} \,$; then as in section~\ref{dd}, the connection 
can be regarded as
an isomorphism $\, \varphi : \D \to \D(I) \,$, and the 
corresponding bimodule is $\, I \otimes_{\O} \D = I\D \,$ with the 
obvious right $\, \D$-module structure and the left $\, \D$-module structure 
defined via $\, \varphi \,$.  Note that the algebra 
$\, \D(I) = I\D I^* \,$ can be identified with the endomorphism ring 
of the right ideal $\, I\D \subseteq \D \,$.  Now (as for any Noetherian 
domain $\, \D \,$) if we are given an 
invertible $\, \D$-bimodule; we may consider it first just as a {\it right} 
$\, \D$-module; it can then be embedded as a right ideal $\, M \,$ in 
$\, \D \,$, and the 
structure of left  $\, \D$-module is given by some isomorphism 
$\, \varphi : \D \to \End_{\D} M \,$.  So Theorem~\ref{CHb} amounts to 
the claim that (in our case) we can always choose $\, M \,$ to be of the form 
$\, I\D \,$, and furthermore that the isomorphism $\, \varphi \,$  then 
restricts to the identity map on $\, \O \,$.  It is in this form that 
the theorem is proved in \cite{CH2}. 

In broad outline, the proof goes as follows.  By \cite{St}, Lemma~4.2, 
we may assume that $\, M \,$ is {\it fat}, that is, $\, M \cap \O \neq 0 \,$. 
The main 
result of \cite{CH1} is that the assignment $\, M \mapsto V := M.\O \,$ 
defines a bijection between the fat right ideals in $\, \D \,$ and  
certain subspaces $\, V \subseteq \O \,$ (called 
``primary decomposable''); and furthermore that $\, \End_{\D} M \,$ then gets 
identified with the algebra
$\, \D(V) := \{ D \in \D(\K) : D.V \subseteq V \} \,$.
Thus the map defining the left $\, \D$-module structure on $\, M \,$ can 
be regarded as an isomorphism $\, \varphi : \D \to \D(V)\,$. 
Concerning $\, V \,$, we need only know that the subalgebra 
$\, \D_0(V) := \D(V) \cap \K \,$ is contained in $\, \O \,$, and that the 
inclusion $\, \D_0(V) \subseteq \O \,$ is a birational isomorphism.
Now we use the main result of \cite{ML}, 
which states that $\,\O\,$ is the {\it unique} maximal abelian ad-nilpotent 
(mad) subalgebra of $\, \D \,$.  Since $\, \D_0(V) \,$ is a mad subalgebra 
of $\, \D(V) \,$, that implies that $\, \varphi \,$ must map $\,\O\,$ 
isomorphically onto $\, \D_0(V) \,$, so that we have a birational 
isomorphism $\, \O \to \D_0(V) \subseteq \O \,$.  Because $\, X \,$ 
is smooth, this must be a genuine isomorphism, so under our assumption 
that $\, \Aut X \,$ is trivial, it must be the identity.  
Thus $\, \D_0(V) = \O \,$, and V is an ideal $\, I \,$ of $\, \O \,$.   
Under the bijection mentioned above, the fat right ideal of $\, \D \,$ 
corresponding to $\, I \,$ is $\, I\D \,$. That completes the proof.

\vspace*{0.2cm}
\scriptsize
\noindent
\textbf{Acknowledgments}. This work was completed during a visit to 
Cornell University; it is a pleasure to thank Cornell Mathematics Department, 
and especially Yuri Berest, for their hospitality.  The support of 
NSF grant DMS 09-01570 is gratefully acknowledged.

\normalsize

\bibliographystyle{amsalpha}

\begin{thebibliography}{A}
%
\bibitem[B]{B}
H.\ Bass. 
\textit{Algebraic K-theory}, 
Benjamin, New York-Amsterdam, 1968.
%
\bibitem[CH1]{CH1}
R. C. Cannings and M. P. Holland, \textit{Right ideals of rings
of differential operators}, 
J. Algebra \textbf{167} (1994), 116--141.
%
\bibitem[CH2]{CH2}
R.~C.~Cannings and M.~P.~Holland, 
\textit{Etale covers, bimodules and differential operators},
Math.~Z.~\textbf{216} (1994), 179--194.
%
%
%
\bibitem[M]{M}
W.\ Messing, \textit{The universal extension of an abelian variety
by a vector group}, Symposia Mathematica, Vol.\ XI (Convegno di 
Geometria, INDAM, Rome) 1972, 359--372.
%
\bibitem[ML]{ML}
L. Makar-Limanov, \textit{Rings of differential operators 
on algebraic curves}, Bull. London Math. Soc. \textbf{21} (1989),
538--540.
%
\bibitem[St]{St}
J. T. Stafford, \textit{Endomorphisms of right ideals of the Weyl algebra},
Trans. Amer. Math. Soc. \textbf{299} (1987), 623--639.
%
\end{thebibliography}

\end{document}